\documentclass[10pt]{amsart}
\usepackage{amsfonts}
\usepackage{amssymb}
\usepackage{amsthm}
\usepackage{amsmath}
\usepackage{amsopn}
\theoremstyle{plain}
\newtheorem{thm}{Theorem}
\newtheorem{prop}[thm]{Proposition}
\newtheorem{lem}[thm]{Lemma}
\newtheorem{cor}[thm]{Corollary}
\newtheorem{question}[thm]{Question}
\newtheorem{remark}[thm]{Remark}

\newtheorem{defn}{Definition}

\usepackage{mathptmx}

\begin{document}

\title{Effective Prime Uniqueness}

\author{Peter Cholak}

\address{Department of Mathematics\\
  University of Notre Dame}

\email{cholak@nd.edu}

\author{Charlie McCoy, C.S.C.}

\address{Department of Mathematics\\University of Portland}

\email{mccoy@up.edu}

\thanks{This work was partially supported by a grant from the Simons
  Foundation (\#315283 to Peter Cholak).}

\maketitle

\begin{abstract}
  Assuming the obvious definitions below, we show that a decidable
  model that is effectively prime is also effectively atomic. This
  implies that two effectively prime (decidable) models are computably
  isomorphic. This is in contrast to the theorem that there are two
  atomic decidable models which are not computably isomorphic. We end
  with a section describing the implications of this result in reverse
  mathematics.
\end{abstract}


\section{Introduction}

Our goal is to explore (decidable) prime models from the perspective
of effective model theory, computability theory, and reverse
mathematics.  In particular, we are interested in the result that
prime models are unique.  Similar results have been found before. In
\cite{MR2529915}, Hirschfeldt, Shore and Slaman looked at classical
results involving prime and atomic models from the perspective of
reverse mathematics.  They left open the analysis of the prime
uniqueness theorem.

We begin with a review of the relevant definitions and results from
classical model theory; see \cite{C-K}.  For this review, fix a
complete theory $T$ of a \textit{countable} language $\mathcal{L}$.

\begin{defn} A formula $\varphi(\vec{x})$ is complete if for every
  other formula $\psi(\vec{x})$, exactly one of the following holds:
  $T \vdash \varphi(\vec{x}) \rightarrow \psi(\vec{x})$ or
  $T \vdash \varphi(\vec{x}) \rightarrow \neg \psi(\vec{x})$.

\end{defn}

\begin{defn} A model $\mathcal{A} \models T$ is atomic if for every
  $\vec{a} \in \mathcal{A}$, there is a complete formula
  $\varphi(\vec{x})$ so that $\mathcal{A} \models \varphi(\vec{a})$.

\end{defn}

\begin{defn} A model $\mathcal{A} \models T$ is prime if for every
  other model $\mathcal{M} \models T$, there is an elementary
  embedding of $\mathcal{A}$ into $\mathcal{M}$.

\end{defn}

\begin{thm} (Atomic Uniqueness) If $\mathcal{A}$ and $\mathcal{B}$ are
  countable atomic models of $T$, then
  $\mathcal{A} \cong \mathcal{B}$.
\end{thm}
 
\begin{proof} Use a back-and-forth construction with complete formulas
  determining how to extend the partial isomorphism.
\end{proof}

\begin{thm} (Atomic $\Rightarrow$ Prime) If $\mathcal{A}$ is a
  countable atomic model of $T$, then $\mathcal{A}$ is prime.

\end{thm}
\begin{proof} Use the ``forth'' half of the back-and-forth argument.
\end{proof}
 
\begin{thm} (Prime $\Rightarrow$ Atomic) If $\mathcal{A}$ is a prime
  model of $T$, then $\mathcal{A}$ is countable and atomic.

\end{thm}
\begin{proof} 
  By the compact theorem $T$ has a countable model.  Therefore
  $\mathcal{A}$ is countable.  Let $\vec{a} \in \mathcal{A}$, and
  consider its type.  For any other model $\mathcal{B}$ of $T$, there
  is an elementary embedding of $\mathcal{A}$ into $\mathcal{B}$, so
  that $\mathcal{B}$ also realizes this type. By the Omitting Types
  Theorem, the type of $\vec{a}$ includes a complete formula.
 
\end{proof}
 
\begin{thm} (Prime Uniqueness) If $\mathcal{A}$ and $\mathcal{B}$ are
  prime models of $T$, then $\mathcal{A} \cong \mathcal{B}$.

  \begin{proof} Immediate by the Atomic Uniqueness and (Prime
    $\Rightarrow$ Atomic).

  \end{proof}

\end{thm}

In \cite{MR2529915}, the authors showed that (Prime $\Rightarrow$
Atomic) holds in $RCA_0$, and that both (Atomic $\Rightarrow$ Prime)
and Atomic Uniqueness are equivalent to $ACA_0$.  Note that in Reverse
Mathematics models are given by their complete diagram. Recall REC is
the canonical model $RCA_0$ where the second order part is just the
collection of computable sets.  Every model in REC is decidable.
Since the classical proof of Prime Uniqueness uses the latter two
theorems, its effective or non-effective content and its place within
Reverse Mathematics are not answered by these results.  In order to
answer these questions, we consider effective analogues of the
classical definitions.

\begin{defn} \label{eff} Let $T$ be a decidable theory and
  $\mathcal{A}$ a decidable model of $T$.

  \begin{enumerate}

  \item The model $\mathcal{A}$ is effectively prime, if for every
    decidable model $\mathcal{M} \models T$, there is a computable
    elementary embedding $f: \mathcal{A} \rightarrow \mathcal{M}$.
    Note that $f$ need not be uniformly computable in $\mathcal{A}$
    and/or $\mathcal{M}$.

  \item The model $\mathcal{A}$ is effectively atomic if there is a
    computable function $g$ that accepts as an input a tuple $\vec{a}$
    from $\mathcal{A}$ (of any length) and outputs a complete formula
    $\varphi(\vec{x})$ so that $\mathcal{A} \models \varphi(\vec{a})$.
    Again $g$ need not be uniformly computable in $\mathcal{A}$.

  \item The model $\mathcal{A}$ is uniformly effectively prime if
    there is a partial computable function $\Phi$ so that, given 
    a decidable $\mathcal{M} \models T$, $\Phi(\mathcal{M})$ halts and
    outputs the code of a computable elementary embedding
    $f: \mathcal{A} \rightarrow \mathcal{M}$.  Again $\Phi$ need not
    be uniformly computable in $\mathcal{A}$.
  \end{enumerate}

\end{defn}


Some observations:

\begin{enumerate}

\item If two decidable models $\mathcal{A}$ and $\mathcal{B}$ of the
  same decidable theory $T$ are both effectively atomic, then the
  classical back and forth construction produces a computable
  isomorphism $f: \mathcal{A} \cong \mathcal{B}$.

\item A modification of the classical proof that atomic implies prime
  shows that effectively atomic implies uniformly effectively prime.


\end{enumerate}

It is essential to note that the results in \cite{MR2529915} are about
decidable, atomic models, \textit{not necessarily} effectively atomic
models.  To understand this distinction, we state a few easily proven
results.

\begin{prop} Let $T$ be a theory in a computable language
  $\mathcal{L}$.  Then the set $\{\phi$ $|$ $\phi$ is a complete
  formula of $T \}$ is $\Pi_1^T$.

\end{prop}

\begin{proof} Check whether, for all other $\psi$ in $\mathcal{L}$,
  exactly one of $T \vdash (\phi \rightarrow \psi)$ or
  $T \vdash (\phi \rightarrow \neg \psi)$ holds.
\end{proof}

Moreover, this result can actually be sharp, even in a case where the
theory $T$ is atomic, as the following result establishes.

\begin{prop} \label{theory} There is a decidable, atomic theory $T$
  for which the set $\{\phi$ $|$ $\phi$ is a complete formula of
  $T \}$ is $\Pi^0_1$-complete.

\begin{proof}
   
  This follows directly from the construction in the proof of Theorem
  2.3 in \cite{MR2529915}.  But we would like to present a
  modification which will be useful in the proof of
  Lemma~\ref{david}.  We will use the same language and the
  collections of Axioms~2,3,4, and 6, from \cite{MR2529915}.  We will
  replace the collection of Axioms~1 with the axioms that $R_i$ are
  pairwise disjoint sets with exactly 2 distinct elements and
  collection of Axioms~5 with if $\Phi_{i,s}(s) \downarrow$ there is
  exactly \emph{one} $x \in R_i$ such that $R_{i,s}(x)$ and exactly
  one $x \in R_i$ such that $\neg R_{i,s}(x)$. Like in
  \cite{MR2529915} we can show that this theory $T$ is recursive, has
  quantifier elimination, is complete, decidable, and atomic.  The
  complete formulas are $R_i(x)$ iff $\Phi_i(i)$ diverges and
  $R_i(x) \wedge R_{i,s}(x)$ and $R_i(x) \wedge \neg R_{i,s}(x)$ iff
  $\Phi_i(i)$ converges in exactly $s$
  steps.  
  $R_i(x)$ is a complete formula of $T$ iff $i \in
  \overline{K}$.  \end{proof}




\end{prop}

This is in contrast with what occurs when a theory has an effectively
atomic model.

\begin{prop} Let $T$ be a decidable theory and $\mathcal{A} \models T$
  a decidable, effectively atomic model.  Then $\{\phi $ $|$ $\phi$ is
  a complete formula of $T \}$ is computable.
\end{prop}

\begin{proof} Let $\phi(\vec{x})$ be a formula with the tuple of free
  variables $\vec{x}$ actually occurring in $\phi$.  First, verify
  that $\phi$ is consistent with $T$.  If so, since $\mathcal{A}$ is a
  model of $T$, there is a tuple of elements $\vec{a} \in \mathcal{A}$
  such that $\mathcal{A} \models \phi(\vec{a})$.  Since $\mathcal{A}$
  is effectively atomic, we can effectively find a complete formula
  $\varphi(\vec{x})$ so that $\mathcal{A} \models \varphi(\vec{a})$
  and hence $T \vdash \varphi(\vec{x}) \rightarrow \phi(\vec{x})$.
  Now, we check if
  $T \vdash \phi(\vec{x}) \rightarrow \varphi(\vec{x})$.  If it does,
  then $\phi(\vec{x})$ is a complete formula, because it implies a
  complete formula.  If it does not, then, since
  $T \not \vdash \phi(\vec{x}) \rightarrow \neg \varphi(\vec{x})$,
  $\phi(\vec{x})$ is not a complete formula.
\end{proof}

Throughout the rest of the section and the next, a theory $T$ will always be a
complete and decidable theory of a computable language
$\mathcal{L}(T)$, and all models will be decidable.  Furthermore, we
assume that all theories and models are presented in such a way that
the associated computable language can always be recovered from the
code for the theory or model.  We will often re-state these facts for
emphasis.

Our main result, whose proof is in Section~\ref{proofsec}, is the
following:

\begin{thm} \label{Main} (Effectively Prime $\Rightarrow$ Effectively
  Atomic) Let $T$ be a decidable theory and $\mathcal{A} \models T$ a
  decidable model.  Then either there is a computable function $h$
  witnessing that $\mathcal{A}$ is effectively atomic; or there is a
  decidable $\mathcal{M} \models T$ such that there is no computable
  elementary embedding of $\mathcal{A}$ into $\mathcal{M}$.

\end{thm}

\begin{cor} \label{main} (Effective Prime Uniqueness) Let $T$ be
  decidable and $\mathcal{A},\mathcal{B} \models T$ be decidable
  models.  Then either there is a computable isomorphism
  $h: \mathcal{A} \cong \mathcal{B}$; or there is a decidable
  $\mathcal{M} \models T$, so that either there is no computable
  elementary embedding of $\mathcal{A}$ into $\mathcal{M}$, or there
  is no computable elementary embedding of $\mathcal{B}$ into
  $\mathcal{M}$.
\end{cor}


By the first observation after Definition~\ref{eff},
Theorem~\ref{Main} implies Corollary~\ref{main}.  By the second
observation, effectively prime, effectively atomic, and uniformly
effectively prime are all equivalent.

Moreover, by looking carefully at the construction in
Section~\ref{proofsec}, we can see that there is actually a greater
degree of uniformity to Theorem ~\ref{Main}, as stated in the next
result.  Note that a code for a decidable model $\mathcal{A}$ is a
Turing machine that computes the complete diagram of $\mathcal{A}$.
Moreover, recall that, by assumption, the presentation for
$\mathcal{A}$ includes the computable language for $\mathcal{A}$.  So
from a decidable model it's theory can be computably recovered.  Thus,
$T$ need not be an input into the functional $\Psi$ below.  However,
the proposition following the result shows that the input of the $e$
is necessary.

\begin{cor} \label{three} There a Turing functional
  $\Psi(\mathcal{A},e)$ such that if $\mathcal{A}$ is a decidable
  model and $T$ is its decidable theory, then either for some $e$,
  $\Psi(\mathcal{A},e)$ is a code for a computable function witnessing
  that $\mathcal{A}$ is effectively atomic; or there is a decidable
  $\mathcal{M} \models T$, such that there is no computable elementary
  embedding of $\mathcal{A}$ into $\mathcal{M}$.

\end{cor}
 
\begin{prop} \label{four} For all $\Psi$, there in an effectively
  atomic $\mathcal{A}$ such that $\Psi(\mathcal{A})$ does not witness
  that $\mathcal{A}$ is effectively atomic.
\end{prop}
  \begin{proof}
    Fix $\Psi$.  By the Recursion Theorem we can assume that we know
    the index $e$ of the model $\mathcal{A}$ we construct.  We will
    work in the language of infinitely unary relations, $U_i$, and our
    model has $\omega$ as its domain.
    
    We are only concerned about the case where $\Psi(e)$ itself is the
    code of a computable function $g$ that accepts tuples of
    $\mathcal{A}$ as inputs and outputs formulas in the language of
    $\mathcal{A}$; in particular, $g$ should accept the 1-tuple $0$.
    
    If at stage $s$, $(\Psi_s(e))_s(0)$ does not halt, then we declare
    $U_s$ to be empty.  If $s$ is the first stage at which
    $(\Psi_s(e))_s(0)$ halts, and it is not (the code of) a formula
    with one free variable, then we declare $U_i$ to be empty for all
    $i$.  If $s$ is the first stage at which $(\Psi_s(e))_s(0)$ halts,
    and it is a formula with one free variable, then let $l$ be the
    least number such that $l \geq s$ and $l \geq j$ for any $j$ where
    $U_j$ is mentioned in this formula. The evens go into $U_{l+1}$
    and the odds stay out. For all $i\leq l$ and $i > l+1$, we declare
    $U_i$ to be empty.

    The resulting $\mathcal{A}$ is the infinite model where either
    there is nothing in any $U_i$; or for $l+1$, there is nothing in
    any $U_i$ for $i \neq l+1$, and $U_{l+1}$ splits the domain into
    evens and odds.  The complete formulas for $1$-types are either
    $x= x$ (for everything) or the pair $U_{l+1}(x)$ (for evens) and
    $\neg U_{l+1}(x)$ (for odds). So $\mathcal{A}$ is effectively
    atomic. However, $(\Psi_s(e))_s(0)$ is certainly not a complete
    formula for the element $0$, because on the $U_j$ mentioned in
    $(\Psi_s(e))_s(0)$, the element $0$ and the element $1$ agree, but
    they disagree on
    $U_{l+1}$. 
  \end{proof}

Hence the ``obvious'' notion of ``uniformly effectively atomic'' is
vacuous.

Finally, we should note that the construction given in the next
section does not depend on knowing ahead of time if the model
$\mathcal{A}$ is infinite or finite.  But it was most likely already
known that Corollary~\ref{three} and Proposition~\ref{four} hold for finite
models.

For instance, for Corollary~\ref{three}, let $e$ code a ``guess'' at
$n$, the size of $|\mathcal{A}| = a_0, a_1, \ldots a_n = \vec{a}$, and
a ``guess'' at the number $l$ of distinct automorphisms of
$\mathcal{A}$ (something less than or equal to $n!$).  Then enumerate
the full diagram of $\mathcal{A}$ until formulas are found that reveal
why the other $n!$ - $l$ permutations of the universe are not
automorphisms.  Let $\Theta(\vec{a})$ be the conjunction of everything
enumerated by this stage. $\Theta(\vec{x})$ is the complete formula
for $\vec{a}$. The complete formula for smaller tuples can be found by
quantifying out certain constants.  Of course, if $e$ codes wrong
guesses about $n$ and $l$ -- or if $\mathcal{A}$ is not, in fact,
finite -- then the formula $\Theta(\vec{x})$ output is not correct.
But if $\mathcal{A}$ is, in fact, finite, then one of the $e$ will
encode correct guesses for $n$ and $l$, and then it outputs a correct
$\Theta(\vec{x})$.

To define a $\Psi(\mathcal{A}, e)$ that works uniformly for both
finite and infinite $\mathcal{A}$, we need the construction given in
the next section.  The above paragraph is intended only to acknowledge
that a much easier functional $\Psi$ works for all finite
$\mathcal{A}$.

For Proposition~\ref{four}, we let our domain be $\{0,1\}$, use $0$ in
place of the evens and $1$ in place of the odds to get a finite atomic
model.

\section{Proof of Theorem~\ref{Main} and its Corollaries }

\label{proofsec}

\subsection{Reference and Conventions} This section builds on the
write-up of the Effective Completeness Theorem given in Harizanov's
survey paper in the Handbook of Recursive Mathematics,
\cite{MR1673621}.  However, we change some of the notations used there
to fit the extra parts of our construction more naturally.

We use a Henkin Construction.  Let
$C = \{c_0, c_1, c_2, \ldots, c_n, \ldots\}$ be the set of new
constants not in the language $\mathcal{L}(T)$.  Let
$\{\sigma_{e}: e \in\omega\}$ be a computable enumeration of the set
of all sentences in the language $\mathcal{L}(T) \cup C$.  (We will
assume some technical things about how these sentences are enumerated,
e.g., about the appearance of the constants of $C$; see below.)

We will effectively enumerate a complete
$(\mathcal{L}(T) \cup C)$-theory $\Gamma \supset T$.  This theory
will, as usual, have Henkin witnesses, so that the desired model
$\mathcal{M}$ has a universe consisting of equivalence classes of the
constants in $C$, where $c_i \equiv c_j$ iff $(c_i = c_j) \in \Gamma$.
Of course, technically, as a model of $T$, our final model is just the
reduct of $\mathcal{M}$ to the language of $\mathcal{L}(T)$.

We computably enumerate $\Gamma$ as
$\{\delta_0, \delta_1, \ldots, \}$, where we enumerate $\delta_s$ at
some point during stage $s$ of the construction.  We denote
$\delta_0 \wedge \ldots \wedge \delta_s$ by $\theta_s(\vec{c}_s$),
where $\vec{c}_s$ is the tuple of all \textit{constants of $C$}
mentioned in the conjunction.

As we enumerate the $\delta_s$ into $\Gamma$, we have to do more than
ensure that $\Gamma$ is a complete diagram that contains $T$ and has
Henkin witnesses.  There are two major additional components to our
construction that must be incorporated.  First, for each computable
function $\Phi$, we try to diagonalize against $\Phi$ being an
elementary embedding of $\mathcal{A}$ into $\mathcal{M}$; if we can
succeed for all $\Phi$, then we will have proven the theorem.  To this
end, we fix, as is standard, a computable enumeration of all
computable functions $\Phi$. Second, for each $\Phi$, if it looks as
though we are failing at all attempts to diagonalize against this
function, then we computably construct, in stagewise fashion, what we
hope will be a computable $h_{\Phi}$ witnessing that $\mathcal{A}$ is
effectively atomic.  When there is no ambiguity, we will drop the
$\Phi$ subscript on $h$.

  Just as in Harizanov's proof of the Effective Completeness Theorem,
  the model $\mathcal{M}$ is really not defined until after the
  stagewise construction is complete, when we can define the
  equivalence classes according to the set $\Gamma$.  Nevertheless
  $\mathcal{M}$ will still be decidable, with either a finite universe
  or an infinite, computable universe, although we cannot say which
  ahead of time.  Therefore, it will be more convenient to conceive of
  the Turing function $\Phi$ as having range not in the universe of
  $\mathcal{M}$ but in the set $C$ of new constants
  $c_1, c_2, \ldots, c_n, \ldots$.  This should not create any
  problems, because using our enumeration of $\Gamma$, there is an
  effective way of converting in either direction between a function
  $\Phi: \mathcal{A} \rightarrow C$ and a function
  $\Phi': \mathcal{A} \rightarrow \mathcal{M}$.  (Given $\Phi$, and
  $a \in \mathcal{A}$, we define $\Phi'(a) := [\Phi(a)]$.  Given
  $\Phi'$, and $a \in \mathcal{A}$, we search, using $\Gamma$, for the
  least element $c$ in the equivalence class $\Phi'(a)$ and define
  $\Phi(a):= c$.) In fact, in our requirements below, we refer to
  $\Phi'$ as the obvious effective translation of $\Phi$.  Finally,
  for convenience, we assume that for all $\Phi$,
  $dom(\Phi) \subseteq |\mathcal{A}|$, the computable universe of
  $\mathcal{A}$.  (That is, we simply ignore whatever is in
  $dom(\Phi) - |\mathcal{A}|$.)

Recall that the standard enumeration of Turing computations of the
form $\Phi_s(a) \downarrow = c$ is such that $a, c < s$.  In our
enumeration of the sentences in $\Gamma$, we will make sure that at
least the constants $c_0, \ldots, c_s$ all appear in $\vec{c}_s$.
This will ensure, simply as a matter of notational convenience, that
no Turing computation produces an \textit{output} (thought of as a
member of $C$) that hasn't been at least technically mentioned
already.  (Again, this is just a matter of convenience.)  Also, we
assume that the enumeration of $\sigma_e$ is such that all of the
constants which appear in $\sigma_e$ are among $c_0, \ldots, c_e$.
Because of how and when we decide to enumerate sentences or their
negations into $\Gamma$, these conventions will ensure that
$\vec{c}_s = c_0, c_1, \ldots c_s$.

Finally, throughout much of the construction, variables are going to
be substituted for constants, and vice-versa, in many formulas; and we
are going to have to consider carefully which constants appearing in a
formula are already in the range of a particular $\Phi_s$ and which
are not.  For instance, $c_1$ may be a constant appearing in the
formula $\varphi$, a fact we denote by writing $\varphi(c_1)$.  If the
variable $x_1$ does not appear in $\varphi$, and we form the new
formula by replacing every appearance of $c_1$ in $\varphi$ with
$x_1$, we will simply write $\varphi(x_1)$ for this new formula.
Similarly, if $\vec{a} = dom(\Phi_s)$, and we break up the tuple
$\vec{c}_s$ into the sub-tuples
$\vec{c}_{s} - \Phi_s(\vec{a}), \Phi_s(\vec{a})$, then when we write
$\theta_s(\vec{c}_s)$ as
$\theta_s(\vec{c}_{s} - \Phi_s(\vec{a}), \Phi_s(\vec{a}))$, we DO NOT
mean to suggest any deep or complex re-arrangement of the constants
within the sentence.  And lastly, as is the convention with free
variables, if we write something like $\sigma_e(\vec{c}_e)$, we mean
to signify that all of the constants of $C$ appearing in $\sigma_e$
are among $\vec{c}_e$, and NOT to signify that all of these constants
do, in fact, appear in $\sigma_e$.

\subsection{A requirement $R_{\Phi}$ requiring attention}

For each Turing function $\Phi$, we have the requirement $R_{\Phi}$:

\begin{itemize}

\item[] $\neg (\Phi': \mathcal{A} \prec \mathcal{M})$; OR

\item[] there is a computable function $h_{\Phi}$ with the following
  properties:

  \begin{enumerate}

  \item the pairs in the graph of $h_{\Phi}$ are of the form
    $(\vec{a}, \varphi(\vec{x}))$, where $\vec{a} \in \mathcal{A}$,
    $\varphi(\vec{x})$ is a complete formula (relative to $T$), and
    $\mathcal{A} \models \varphi(\vec{a})$.

  \item for each $\vec{a} \in \mathcal{A}$, $\vec{a}$ is a sub-tuple
    of a tuple $\vec{a}'$ that appears in the domain of $h_{\Phi}$.

  \end{enumerate}
   
\end{itemize}

For each requirement $R_{\Phi}$, we refer to the \textit{index} of the
requirement and the index of $\Phi$ interchangeably.  As usual, one
requirement is higher priority than another if its index is lower.
Recall from the previous section that $\mathcal{M}$ is a reduct from a
Henkin construction built with new constants $c_0, c_1 \ldots $ and 
$\Phi'(a) = [ \Phi(a) ] $.

Note: Because of the conditions above for the function $h_{\Phi}$,
from $h_{\Phi}$ we could automatically construct a computable function
$g$ that accepts any tuple $\vec{a}$ from $\mathcal{A}$ and outputs a
complete formula satisfied by $\vec{a}$.  Given $\vec{a}$, by the
second condition, find a tuple $\vec{a}'$ in the domain of $h$ with
$\vec{a} \subseteq \vec{a}'$, and let $h(\vec{a}')$ be
$\varphi(\vec{x}')$.  By the first condition,
$\mathcal{A} \models \varphi(\vec{a}')$.  Consider $\varphi(\vec{a}')$
as $\varphi(\vec{a}, \vec{a}' - \vec{a})$, let $\vec{x}$ be a tuple of
new variables of the same length as $\vec{a}$, and let $\vec{y}$ be a
tuple of new variables of the same length as $\vec{a}' - \vec{a}$.
Then it is quickly verified that
$\exists \vec{y} \varphi(\vec{x}, \vec{y})$ is a complete formula
satisfied by $\vec{a}$.

\begin{defn} A requirement of the form $R_{\Phi}$ is
  \textit{completely satisfied by stage $s$} if AT LEAST ONE of the
  following two conditions holds:

  \begin{enumerate}

  \item $\Phi_s$ is not 1-1; OR

  \item If $\vec{a} = dom(\Phi_s)$, and we look at
    $\theta_{s}(\vec{c}_{s})$ as
    $\theta_s(\vec{c}_{s} - \Phi_s(\vec{a}), \Phi_s(\vec{a}))$, and
    $\vec{y}$ is a tuple of new variables (not appearing among the
    variables in $\theta_{s}(\vec{c}_{s})$) of the same length as
    $\vec{c}_{s} - \Phi_s(\vec{a})$, then
    $\mathcal{A} \not \models \exists \vec{y} \theta_s(\vec{y},
    \vec{a})$.
    (Note: the substitution of $\vec{a}$ for $\Phi_s(\vec{a})$ is
    unambiguous, because, if the first condition does not hold, then
    $\Phi_s$ is assumed to be 1-1.)

  \end{enumerate}

\end{defn}

It is important for the reverse mathematics to note that a requirement
$R_\Phi$ being completely satisfied by stage $s$ is a computable
condition.  Therefore, a requirement $R_\Phi$ eventually becoming
completely satisfied is a $\Sigma_1$ condition.

\begin{defn} The stage $s$ approximation to $h_{\Phi}$ is denoted by
  $h_{\Phi, s}$ (or just $h_s$, if we're dropping the function
  subscripts).  To initialize the stage $s - 1$ approximation
  $h_{s-1}$ at stage $s$ simply means to re-define it to be equal to
  $\emptyset$. 

\end{defn}

Since our construction informally involves substages, it might be the
case $h_{\Phi, s-1}$ is initialized at a substage of stage $s$ and at a
later substage of stage $s$ redefined to be nonempty.

\begin{defn} A requirement of the form $R_{\Phi}$ requires attention
  at stage $s$ if

  \begin{enumerate}

  \item $R_{\Phi}$ is not completely satisfied by stage $s$;

  \item $\Phi_s$ has converged on at least the input $a_0$, and one of
    the following is true:

    \begin{itemize}

    \item $h_{\Phi, s-1} = \emptyset$ or has been initialized at this
      stage $s$, and $\Phi_s(a_0)\downarrow$; OR

    \item $\Phi_s$ has converged on $k$ inputs in the domain of
      $\mathcal{A}$, and $T \vdash \tau$, where $\tau$ expresses that
      there exist $k$ distinct elements and there don't exist $k+1$
      distinct elements; OR
    
    \item $h_{\Phi, s-1} \not = \emptyset$ and has not been
      initialized at this stage $s$; and the domain of $\Phi_s$
      contains an initial segment of the universe of $\mathcal{A}$
      that includes all of the tuples in $dom(h_{\Phi,s-1})$ and at least
      one more element.

    \end{itemize}

  \end{enumerate}

\end{defn}

\subsection{Construction}

\noindent Stage 0:

$\delta_{0} := (c_0 = c_0)$.  All functions
$h_{\Phi, 0} := \emptyset$.

\vspace{1em}

\noindent Stage $s = 2k + 1$ for $k \in \omega$ (Henkin witness
requirement):

Case 1: $\delta_{k} = \exists x \gamma (x) \wedge \tau$, where $\tau$
is a possibly empty conjunction of sentences of the form
$(c_i = c_i)$.  By convention, we know that the first element of $C$
that does not appear in $\theta_{s-1}(\vec{c}_{s-1})$ is $c_s$.
Define $\delta_{s} := \gamma(c_s) \wedge (c_s = c_s)$.

Case 2: Otherwise.  Define $\delta_{s} := (c_s = c_s)$.

\vspace{1em}

\noindent Stage $s = 2k + 2$ for $k \in \omega$ (Completeness of the
diagram requirement):

This portion of the construction, dedicated to the determination of
$\delta_s$ at a positive even stage, employs an algorithm with a
``loop'' structure (that always terminates; see below).  Notice that
each step of the algorithm is computable.

Let $e$ be the least $e$ for which we have not explicitly decided
whether to add $\sigma_{e}$ or $\neg \sigma_{e}$ to $\Gamma$; i.e., at
no previous stage $t$ did
$\delta_{t} := \sigma_{e} \wedge (c_t = c_t)$ or
$\delta_{t} := \neg \sigma_{e} \wedge (c_t = c_t)$.  We will work to
make this determination at this stage, unless the complete
satisfaction of a higher priority requirement $R_{\Phi}$ forces us to
decide a different statement.

\subsubsection{Algorithm}

\begin{enumerate}

\item Set $\sigma^{*} := \sigma_e$ and $i^{*} := e$.

\item Determine if the following is true: for $\gamma = \sigma^{*}$ or
  for $\gamma = \neg \sigma^{*}$, if $\vec{x}$ is a tuple of new
  variables (not appearing among the variables in
  $\theta_{s-1}(\vec{c}_{s-1}) \wedge \gamma(\vec{c}_{s-1})$) of the
  same length as $\vec{c}_{s-1}$, then
  $T \vdash \forall \vec{x}[(\theta_{s-1}(\vec{x}) \rightarrow
  \gamma(\vec{x})]$.

\item If it is true for either $\gamma = \sigma^{*}$ or for
  $\gamma = \neg \sigma^{*}$, then only this $\gamma$ is consistent
  with $T$ and $\theta_{s-1}(\vec{c}_{s-1})$.  Define
  $\delta_{s} := \gamma \wedge (c_s = c_s)$ and exit the algorithm.
  Otherwise, then each of $\sigma^{*}$ and $\neg \sigma^{*}$ is
  consistent with $T$ and $\theta_{s-1}(\vec{c}_{s-1})$, so proceed to
  the next step.

\item Determine if there is any requirement $R_{\Phi}$ with index
  $\leq i^{*}$ that has not been completely satisfied up to this point
  in stage $s$. (Recall that a requirement can become completely
  satisfied at a given stage simply by the computation revealing
  $\Phi$ is not 1-1.  Also recall that a requirement being completely
  satisfied by stage $s$ is a computable condition.)

\item If there is no such requirement, then define
  $\delta_s := \sigma^{*} \wedge (c_s = c_s)$, and exit the algorithm.
  Otherwise, proceed to the next step.

\item \textit{For each} function $\Phi$ associated with a requirement
  that has not been completely satisfied and has index $\leq i^{*}$,
  complete the following analysis:

  \begin{itemize}

  \item Let $\vec{a} = dom(\Phi_s)$.  (Recall that, by the conventions
    we mentioned above, $ran(\Phi_s) \subseteq \vec{c}_{s-1}$ and all
    of the constants appearing in $\sigma_{e}$ are among
    $\vec{c}_{s-1}$, as well.)

  \item Determine if one of the following conditions hold:

    \begin{enumerate}

    \item For $\gamma = \sigma^{*}$ or for $\gamma = \neg \sigma^{*}$,
      if we look at $\theta_{s-1}(\vec{c}_{s-1}) \wedge \gamma$ as
      $\rho(\vec{c}_{s-1} - \Phi_s(\vec{a}), \Phi_s(\vec{a}))$, and if
      $\vec{y}$ is a tuple of new variables (not appearing among the
      variables in $\theta_{s-1}(\vec{c}_{s-1}) \wedge \gamma$) of the
      same length as $\vec{c}_{s-1} - \Phi_s(\vec{a})$, then
      $\mathcal{A} \not \models \exists \vec{y} \rho(\vec{y},
      \vec{a})$.

      (\textbf{COMMENT}: Each of the sentences, $\sigma^{*}$ and
      $\neg \sigma^{*}$, are consistent with $T$ and
      $\theta_{s-1}(\vec{c}_{s-1})$, but one of them would make it
      impossible for $\Phi'$ to be an elementary embedding.)

    \item The previous condition does not hold, but for
      $\gamma = \sigma^{*}$ or for $\gamma = \neg \sigma^{*}$, if

      \begin{itemize}

      \item we look at $\theta_{s-1} \wedge \gamma$ as the formula
        $\theta_{s-1}(\vec{c}_{s-1} -
        \Phi_{s}(\vec{a}),\Phi_{s}(\vec{a})) \wedge
        \gamma(\vec{c}_{s-1} - \Phi_s(\vec{a}), \Phi_s(\vec{a}));$

      \item $\vec{x}$ is a tuple of new variables of the same length
        as $\Phi_s(\vec{a})$;

      \item $\vec{y}$ is a tuple of new variables of the same length
        as $\vec{c}_{s-1} - \Phi_{s}(\vec{a})$,

      \end{itemize}

      then
      $T \vdash \exists \vec{x} [\exists \vec{y}(\theta_{s-1}(\vec{y},
      \vec{x})) \wedge \forall\vec{y}(\theta_{s-1}(\vec{y}, \vec{x})
      \rightarrow \gamma(\vec{y}, \vec{x}))]$.

      (\textbf{COMMENT}: Since condition a) doesn't hold, we know
      that, based on what has been declared so far in $\theta_{s-1}$,
      each of $\sigma^{*}$ and $\neg \sigma^{*}$ is consistent with
      $\vec{a} \mapsto \Phi'(\vec{a})$ as part of a potential
      elementary embedding.  However, in this case, $T$ guarantees
      that there is a tuple $\vec{x}$ of elements which satisfies the
      existential statements necessary to be consistent with
      $\theta_{s-1}$, but which can accommodate only one of
      $\sigma^{*}$ or $\neg \sigma^{*}$.  Therefore, defining the
      non-trivial part of $\delta_{s}$ to be
      $\gamma \wedge \forall \vec{y} (\theta_{s-1}(\vec{y},
      \Phi_{s}(\vec{a})) \rightarrow \gamma (\vec{y},
      \Phi_s(\vec{a})))$.
      This would make it impossible for $\Phi'$ to be an elementary
      embedding since now the types of $\vec{a}$ in $\mathcal{A}$ and
      $\Phi'(\vec{a})$ in $\mathcal{M}$ are different.

    \end{enumerate}

  \end{itemize}

\item If all of the functions $\Phi$ that are considered don't satisfy
  any of the above conditions, then define
  $\delta_{s} := \sigma^{*} \wedge (c_s = c_s)$, and exit the
  algorithm.  Otherwise, proceed to the next step.

\item REDEFINE $i^{*}$ to be the index of the highest priority
  requirement that was considered and satisfies one of conditions a)
  or b) under the second bullet of step (6).  In the rest of the
  steps, $\Phi$ refers specifically to the Turing function for this
  requirement.

\item If the function satisfies Step (6) condition a), then, for the
  appropriate $\gamma$ that makes the condition satisfied (either
  $\sigma^{*}$ or $\neg \sigma^{*}$, and there is no ambiguity which),
  define $\delta_{s} := \gamma \wedge (c_s = c_s)$; and exit the
  algorithm.  Otherwise, proceed to the next step.

\item If the function satisfies condition b), then it is possible that
  the satisfaction could be due to either $\gamma = \sigma^{*}$ or
  $\gamma = \neg \sigma^{*}$; if this is the case, show (arbitrary)
  preference for $\gamma = \sigma^{*}$; if not, then the $\gamma$ that
  makes the condition satisfied is unambiguous.  Now, for this
  $\gamma$, REDEFINE
  $\sigma^{*} := \gamma \wedge \forall \vec{y} (\theta_{s-1}(\vec{y},
  \Phi_{s}(\vec{a})) \rightarrow \gamma (\vec{y}, \Phi_s(\vec{a})))$.
  And, with this new index $i^{*}$ and this new $\sigma^{*}$, return
  to the second step of the algorithm.

  (\textbf{COMMENT}: Why redefine $\sigma^{*}$ instead of just
  defining $\delta_s$ to be the conjunction of this new $\sigma^{*}$
  and $(c_s = c_s)$?  If $\delta_s$ were defined in this way, then the
  respective requirement would be satisfied; however, because this
  $\delta_s$ was not analyzed in the earlier steps of the algorithm,
  it is possible that, in adding this $\delta_s$, as opposed to the
  negation of the non-trivial part, an opportunity was missed to
  completely satisfy a higher priority requirement.  Thus, the need to
  redefine $\sigma^{*}$ and restart the algorithm.  Note that if no
  higher priority requirement meets one of the conditions of Step (6)
  in the next iteration, then in this next iteration the algorithm we
  get past Step (3) since the new $\sigma^*$ is a stronger consistent
  clause than the old $\sigma^*$; we will exit at Step (9); $\delta_s$
  will be defined as suggested, and the respective requirement will be
  completely satisfied.)

\end{enumerate}

Notice that for each successive loop through the algorithm, the index
$i^{*}$ is strictly less than it was before, so the algorithm must
terminate, and $\delta_s$ is well-defined.

\subsubsection{Definition/Construction of the stage s approximations
  to the potential isomorphisms}

If $R_{\Phi}$ is the highest priority requirement (with index less
than or equal to $e$) that was not completely completely satisfied at
stage $s-1$ and is completely satisfied during this stage $s$, then
initialize all functions $h_{s-1}$ associated with all lower priority
requirements.  If there is no such requirement, then simply initialize
all functions $h_{s-1}$ associated with requirements $R_{\Phi}$ with
index greater than or equal to $e$.

As the final part of the construction at positive even stages, we
define $h_{\Phi, s}$ on the requirements $R_{\Phi}$ that still require
attention at stage $s$ (even after our work at stage $s$ so far).  We
will focus on one of these and refer to it as $h_s$ from now on.  (But
again, we would do this work for \textit{every} $R_{\Phi}$ that still
requires attention at stage $s$, which, by definition, is a finite
number of requirements.)

Let $\vec{a} = dom(\Phi_s)$ (thought of as an ordered tuple, not just
a set).  By the assumption that $R_{\Phi}$ requires attention at this
stage, if $h_{s-1}$ had been initialized above during stage $s$ then
$\vec{a}$ contains at least $a_0$; if we know at this stage that
$\mathcal{A}$ is finite model of size $k$ then $\vec{a}$ is the entire
universe of $\mathcal{A}$; or $\vec{a}$ contains an initial segment of
the universe of $\mathcal{A}$ that includes all tuples in
$dom(h_{s-1})$ and at least one more element.

Recall that we are automatically conceiving of $\Phi_s(\vec{a})$ as
being constants from $C$ and among $\vec{c}_s$.  Consider the sentence
$\theta_{s}(\vec{c}_s)$.  We look at $\theta_{s}(\vec{c}_s)$ as
$\theta_s(\vec{c}_s - \Phi_s(\vec{a}), \Phi_s(\vec{a}))$.  Let
$\vec{y}$, $\vec{x}$ be two new, disjoint tuples of variables (not
appearing among the variables of $\theta_s$) of the same length as
$\vec{c}_s - \Phi_s(\vec{a})$, $\Phi_s(\vec{a})$, respectively.
Define
$h_s(\vec{a}) := \phi(\vec{x}) = \exists\vec{y} \theta_s(\vec{y},
\vec{x})$.
(Clearly, $\mathcal{A} \models \phi(\vec{a})$, because $R_{\Phi}$ has
not been completely satisfied).  The majority of the Verification
subsection below is devoted to proving that -- for any requirement
$R_{\Phi}$ to receive attention infinitely often, and after finitely
much initialization due to higher priority requirements has stopped --
the formulas $\phi(\vec{x})$ are complete.)

Finally, for all other functions $h_{\hat{\Phi}}$ associated with
other requirements that have not already been initialized at this
stage $s$, let $h_{\hat{\Phi}, s} := h_{\hat{\Phi}, s - 1}$.

This concludes the construction.

\subsection{Verification}

\begin{lem} $\mathcal{M}$ is decidable and $\mathcal{M} \models T$.

  \begin{proof} The construction is an expansion on the standard
    Henkin construction.  All of the components that guarantee the
    claim of the lemma are included.  First, the construction
    constructs a complete theory $\Gamma$ in the expanded language by
    eventually adding $\sigma_e$ or $\neg \sigma_e$ (with a trivial
    conjunct of the form $(c_s = c_s)$ appended) to $\Gamma$.  It is
    true that, even if $\sigma_{e}$ is the original sentence
    considered at a particular even stage $s$, the above algorithm,
    because of Step (6) condition b), might redefine $\delta_s$ to be
    a sentence that implies neither $\sigma_e$ nor $\neg \sigma_e$.
    Now, without any such delays, the sentence $\sigma_e$ would be
    decided by stage $2e + 2$ at the latest.  However, the decision
    can be delayed only by $R$ requirements with index $< e$.
    Therefore, stage $s = 4e + 4$ provides an upper bound on the stage
    by which $\sigma_e$ or $\neg \sigma_e$ (with a trivial conjunct
    appended) is included in $\Gamma$.

    Second, the algorithm employed at even stages, which is not part
    of the standard Henkin construction, always terminates, and it
    preserves consistency with $T$ throughout.  Third, the odd stages
    simply guarantee the existence of Henkin witnesses.  Fourth, as in
    the standard Henkin construction, elements of the model are
    equivalence classes of constant symbols.

    Finally, the definitions of the parts of functions $h_{\Phi,s}$ is
    an additional component of our construction, but this part of the
    construction does not affect choices in how we build $\mathcal{M}$
    and the complete theory $\Gamma$.
 
  \end{proof}

\end{lem}

\begin{lem} If every requirement $R_{\Phi}$ requires attention only
  finitely often, then $\mathcal{A}$ is not embeddable by a computable
  embedding into $\mathcal{M}$.

  \begin{proof} Assume every requirement requires attention only
    finitely often.  By definition, there are only two reasons that a
    requirement $R_{\Phi}$ stops requiring attention by stage $s$.
    First, because $R_{\Phi}$ becomes completely satisfied by $s$, so
    one of the following is true:

    \begin{itemize}

    \item the corresponding
      $\Phi': \mathcal{A} \rightarrow \mathcal{M}$ is not 1-1; OR

    \item for some tuple $\vec{a} \in \mathcal{A}$ and some forumla
      $\varphi(\vec{x})$, $\mathcal{A} \models \varphi(\vec{a})$ and
      $\mathcal{M} \models \neg \varphi(\Phi'(\vec{a}))$.

    \end{itemize}

    \noindent (See the above section on conventions and facts
    regarding the connection between $\Phi$ and $\Phi'$.)

    Second, because $dom(\Phi)$ does not include the universe of
    $\mathcal{A}$, and hence $\Phi'$ does not include the universe of
    $\mathcal{A}$.

    Now, as the section on conventions explained, every computable
    function $f: \mathcal{A} \rightarrow \mathcal{M}$ is equal to
    $\Phi'$ for some $\Phi: \mathcal{A} \rightarrow {C}$.
    Therefore, if \textit{every requirement} $R_{\Phi}$ stops
    requiring attention by some stage $s$, then every computable
    function from $\mathcal{A}$ to $\mathcal{M}$ fails to be an
    elementary embedding.

  \end{proof}

\end{lem}

\begin{remark}\label{remark}
  Therefore, for the rest of this verification, we assume that there
  is a requirement $R_{\Phi}$ and stages $s^{*} \leq s$ with the
  following three properties:

  \begin{itemize}

  \item $R_{\Phi}$ requires attention infinitely often.

  \item $s^{*}$ is the \textit{least stage} $t$ with the following
    three properties:

    \begin{itemize}

    \item $t >$ the index of $\Phi$

    \item for each stage $u \geq t$, it is NOT the case that a
      requirement $R_{\hat{\Phi}}$ of priority higher than that of
      $R_{\Phi}$ first becomes completely satisfied at $u$;

    \item for each $e$ less than or equal to the index of $\Phi$, the
      algorithm in sub-subsection 2.3.1 has explicitly added
      $\sigma_e$ or $\neg \sigma_e$ to $\Gamma$ before stage $t$.

    \end{itemize}

  \item $s$ is the first stage $\geq s^{*}$ so that $R_{\Phi}$
    requires attention at $s$.

  \end{itemize}

\end{remark}

This requirement and these stages will be of particular importance as
we state and prove the uniform version of this theorem below.

With this requirement $R_{\Phi}$ and these stages $s^{*}$ and $s$
\textit{fixed}, we must prove that
$h_{\Phi} = \bigcup_{t \geq s}h_{\Phi, t}$ has the properties stated
near the beginning of Subsection 2.2.  We will refer to this function
simply as $h$ from now on, and its stage $t$ approximation as $h_t$.
The following long lemma will essentially complete this proof.  Recall
the notation from subsection 2.1 that $\theta_w$ is the conjunction of
all sentences of $\Gamma$ enumerated by the end of stage $w$.

\begin{lem} For each stage $t \geq s$ for which $R_{\Phi}$ requires
  attention, we recall or consider the following notational
  conventions:

  \begin{enumerate}

  \item $\vec{a}_t = dom(\Phi_t)$;

  \item $\vec{x}_t$ is a tuple of new variables (i.e., not appearing
    in $\theta_t$) of the same length as $\vec{a}_t$ (which is the
    same length as $\Phi(\vec{a}_t)$ since $\Phi$ is 1-1);

  \item for each $u \geq t$, $\vec{y}_u$ is a tuple of new variables
    (i.e., not appearing in $\theta_u$) of the same length as
    $\vec{c}_u - \Phi_t(\vec{a}_t)$;

  \item
    $h_t(\vec{a}_t) = \phi(\vec{x}_t) = \exists\vec{y}_t
    \theta_t(\vec{y}_t, \vec{x}_t)$;

  \item for each $u \geq t$, we consider
    $\theta_u(\vec{c}_u) = \theta_u(\vec{c}_u - \Phi_t(\vec{a}),
    \Phi_t(\vec{a}))$,
    and we assume (making trivial changes, if necessary) that
    $\theta_u$ does not use any of the variables in the tuple
    $\vec{x}_t$.

  \end{enumerate}
    
  Then for all $u \geq t$,
  $\mathcal{A} \models \exists \vec{y}_u \theta_u(\vec{y}_u,
  \vec{a}_t)$
  and
  $T \vdash \phi(\vec{x}_t) \rightarrow \exists \vec{y}_u
  \theta_u(\vec{y}_u, \vec{x}_t)$.
  
  \noindent (Note: in (3), (4), and the conclusion of the lemma, the
  different subscripts $u$ and $t$ are intentional.)

  \begin{proof} Let $t \geq s$ be a stage where $R_{\Phi}$ requires
    attention.  

    For all $u \geq t$, the first part of the statement must be true.
    Assume otherwise.  Then, since $u \geq t$,
    $dom(\Phi_t) \subseteq dom(\Phi_{u})$; and so, it would certainly
    be the case that if $\vec{a} = dom(\Phi_{u})$, and we look at
    $\theta_{u}(\vec{c}_{u})$ as
    $\theta_u(\vec{c}_{u} - \Phi_{u}(\vec{a}), \Phi_{u}(\vec{a}))$,
    and $\vec{y}$ is a tuple of new variables (not appearing among the
    variables in $\theta_{u}(\vec{c}_{u+1})$) of the same length as
    $\vec{c}_{u} - \Phi_{u}(\vec{a})$, then
    $\mathcal{A} \not \models \exists \vec{y} \theta_{u}(\vec{y},
    \vec{a})$.
    Therefore, $R_{\Phi}$ would be completely satisfied, and would no
    longer receive attention.  Therefore, for all $u \geq t$,
    $\mathcal{A} \models \exists \vec{y}_u \theta_u(\vec{y}_u,
    \vec{a}_t)$.
  
    We prove the second part of the statement by induction on
    $u \geq t$.  For $u = t$, of course, $\phi(\vec{x}_t)$ and
    $\exists \vec{y}_u \theta_u(\vec{y}_u, \vec{x}_t)$ are exactly the
    same formula, so the statement is obviously true.  Assume that for
    all $u'$ with $t \leq u' \leq u$,
    $T \vdash \phi(\vec{x}_t) \rightarrow \exists \vec{y}_{u'}
    \theta_{u'}(\vec{y}_{u'}, \vec{x}_t)$.
    We must show that
    $T \vdash \phi(\vec{x}_t) \rightarrow \exists \vec{y}_{u+1}
    \theta_{u+1}(\vec{y}_{u+1}, \vec{x}_t)$.

    Recall that the statement $\theta_{u+1}(\vec{c}_{u+1})$ is just
    the statement $\theta_{u}(\vec{c}_{u}) \wedge \delta_{u+1}$, where
    $\delta_{u+1}$ is the sentence added at stage $u+1$ of the
    construction given in subsection 2.3.  The form of this sentence
    $\delta_{u+1}$ depends on the number $u+1$.  We consider the
    cases.
    
    \underline{Case 1a)} $u+1 = 2k + 1$ for some $k \in \omega$, and
    $\delta_k = \exists x \gamma(x) \wedge \tau$, where $\tau$ is a
    conjunction of sentences of the form $(c_i = c_i)$.  Then
    $\delta_{u+1} = \gamma(c_{u+1}) \wedge (c_{u+1} = c_{u+1})$.
    Since $u > k$, the sentence $\delta_k$ is already included as one
    of the conjuncts of $\theta_u(\vec{c}_u)$.  Therefore,
    $\exists \vec{y}_{u} \theta_{u}(\vec{y}_{u}, \vec{x}_t)$ has the
    form
    $\exists \vec{y}_{u}[\ldots \wedge \exists x \gamma(x) \wedge
    \ldots]$,
    where whatever elements of $\vec{c}_k (\subseteq \vec{c}_{u})$
    appearing in $\gamma(x)$ have been replaced by the corresponding
    elements of $\vec{y}_{u}$ or $\vec{x}_{t}$, according to our
    normal substitution conventions.  In particular, we assume that
    the variable $x$ in $\gamma(x)$ is not one of the variables in the
    tuple $\vec{x}_t$.

    Similarly, since
    $\theta_{u+1}(\vec{c}_{u+1})= \theta_{u}(\vec{c}_{u}) \wedge
    \delta_{u+1}$,
    and $\delta_{u+1} = \gamma(c_{u+1}) \wedge (c_{u+1} = c_{u+1})$, \\
    $\exists \vec{y}_{u+1} \theta_{u+1}(\vec{y}_{u+1}, \vec{x}_t)$ has
    the form
    $\exists \vec{y}_{u}\exists y_{u+1}[\ldots \wedge \exists x
    \gamma(x) \wedge \ldots \wedge \gamma(y_{u+1}) \wedge (y_{u+1} =
    y_{u+1})]$,
    where all other substitutions of the variables of $\vec{y}_{u}$
    and $\vec{x}_{t}$ in the two appearances of $\gamma$ are exactly
    the same.  Furthermore, by our conventions, neither $\gamma(x)$
    nor any of the other conjuncts in $\theta_{u}$ makes any mention
    of $c_{u+1}$.  Therefore, the formula
    $\exists \vec{y}_{u+1} \theta_{u+1}(\vec{y}_{u+1}, \vec{x}_t)$ and
    the formula
    $\exists \vec{y}_{u} \theta_{u}(\vec{y}_{u}, \vec{x}_t)$ are
    logically equivalent.  Since
    $T \vdash \phi(\vec{x}_t) \rightarrow \exists \vec{y}_{u}
    \theta_{u}(\vec{y}_{u}, \vec{x}_t)$,
    $T \vdash \phi(\vec{x}_t) \rightarrow \exists \vec{y}_{u+1}
    \theta_{u+1}(\vec{y}_{u+1}, \vec{x}_t)$.

    \underline{Case 1b)}: $u+1 = 2k +1$ for some $k \in \omega$, but
    $\delta_k$ does not have the above form of an existential sentence
    (with a trivial $\tau$ attached).  In this case $\delta_{u+1}$ is
    just the trivial sentence $(c_{u+1} = c_{u+1})$, so again,
    trivially, the formula
    $\exists \vec{y}_{u+1} \theta_{u+1}(\vec{y}_{u+1}, \vec{x}_t)$ and
    the formula
    $\exists \vec{y}_{u} \theta_{u}(\vec{y}_{u}, \vec{x}_t)$ are
    logically equivalent.  Therefore,
    $T \vdash \phi(\vec{x}_t) \rightarrow \exists \vec{y}_{u+1}
    \theta_{u+1}(\vec{y}_{u+1}, \vec{x}_t)$.
  
    \underline{Case 2}: $u+1 = 2k + 2$.  Therefore, $\delta_{u+1}$ is
    determined by the algorithm in sub-subsection 2.3.1.  That is,
    $\delta_{u+1} = \pm\sigma^{*} \wedge (c_{u+1} = c_{u+1})$ for
    $\sigma^{*}$ relative to the last iteration of the algorithm at
    stage $u+1$.  For the rest of this proof, we refer to the
    non-trivial part of $\delta_{u+1}$ as $\gamma$; i.e.,
    $\gamma = \sigma^{*}$ or $\gamma = \neg \sigma^{*}$.  Note that
    $c_{u+1}$ does not appear in $\gamma$. 

    If the algorithm at stage $u+1$ at this last iteration exits at
    Step 3, then it is the case that
    $T \vdash \forall\vec{z}[\theta_u(\vec{z}) \rightarrow
    \gamma(\vec{z})]$.
    Therefore, since by induction hypothesis,
    $T \vdash \phi(\vec{x}_t) \rightarrow \exists \vec{y}_{u}
    \theta_{u}(\vec{y}_{u}, \vec{x}_t)$,
    and
    $\theta_{u+1}(\vec{y}_{u+1}, \vec{x}_t) = \theta_u(\vec{y}_u,
    \vec{x}_t) \wedge \gamma(\vec{y}_u, \vec{x}_t) \wedge (y_{u+1} =
    y_{u+1})$,
    $T \vdash \phi(\vec{x}_t) \rightarrow
    \exists\vec{y}_{u+1}\theta_{u+1}(\vec{y}_{u+1}, \vec{x}_t)$.

    It cannot be the case that the algorithm exits at Step 9 for the
    sake of $\Phi$, for then $\Phi$ would be completely satisfied and
    would stop receiving attention.

    Finally, for the rest of this case, we assume, in order to obtain
    a contradiction, that
    $T \not \vdash [\phi(\vec{x}_t) \rightarrow
    \exists\vec{y}_{u+1}\theta_{u+1}(\vec{y}_{u+1}, \vec{x}_t)]$.
    That is, we assume that
    $$T \vdash \exists \vec{x}_t[\phi(\vec{x}_t) \wedge
    \forall\vec{y}_{u+1}(\neg \theta_{u+1}(\vec{y}_{u+1},
    \vec{x}_{t}))].$$
    Again, since
    $\theta_{u+1}(\vec{y}_{u+1}, \vec{x}_t) = \theta_u(\vec{y}_u,
    \vec{x}_t) \wedge \gamma(\vec{y}_u, \vec{x}_t) \wedge (y_{u+1} =
    y_{u+1})$,
    $\forall\vec{y}_{u+1}(\neg \theta_{u+1}(\vec{y}_{u+1},
    \vec{x}_{t}))$
    is logically equivalent to
    $\forall\vec{y}_{u}(\neg \theta_u(\vec{y}_u, \vec{x}_t) \vee \neg
    \gamma(\vec{y}_u, \vec{x}_t))$,
    which is logically equivalent to
    $\forall\vec{y}_{u}( \theta_u(\vec{y}_u, \vec{x}_t) \rightarrow
    \neg \gamma(\vec{y}_u, \vec{x}_t))$.
    Therefore,
    $T \vdash \exists\vec{x}_t[\phi(\vec{x}_t) \wedge
    \forall\vec{y}_{u}( \theta_u(\vec{y}_u, \vec{x}_t) \rightarrow
    \neg \gamma(\vec{y}_u, \vec{x}_t))]$.
    Moreover, by induction hypothesis,
    $T \vdash \phi(\vec{x}_t) \rightarrow \exists \vec{y}_{u}
    \theta_{u}(\vec{y}_{u}, \vec{x}_t)$.
    And so,
    $$T \vdash \exists\vec{x}_t[\exists \vec{y}_{u}
    \theta_{u}(\vec{y}_{u}, \vec{x}_t) \wedge \forall\vec{y}_{u}(
    \theta_u(\vec{y}_u, \vec{x}_t) \rightarrow \neg \gamma(\vec{y}_u,
    \vec{x}_t))].$$

    Now, except for the use of $\neg \gamma$ instead of $\gamma$, this
    last statement is almost exactly what appears at stage $u+1$ in
    Condition b) under the second bullet point of Step (6) of the
    algorithm, which is
    $T \vdash \exists \vec{x} [\exists \vec{y}(\theta_{u}(\vec{y},
    \vec{x})) \wedge \forall\vec{y}(\theta_{u}(\vec{y}, \vec{x})
    \rightarrow \gamma(\vec{y}, \vec{x}))]$.
    However, we have to be careful, because the length of the tuples
    is not correct; i.e., at stage $u+1$, the length of $\vec{x}$
    mentioned in the algorithm is the same as the length of the range
    of $\Phi_{u+1}$, and the length of $\vec{y}$ mentioned in the
    algorithm is the same as the length of $(\vec{c}_u -$ the range of
    $\Phi_{u+1})$.  Notice, because $u \geq t$, that the length of
    $\vec{x}_t$ is less than or equal to that of $\vec{x}$ in the
    algorithm, so the length of $\vec{y}_u$ is greater than or equal
    to that of $\vec{y}$ in the algorithm.  Nevertheless, the
    following paragraph establishes that, indeed,
    $T \vdash \exists \vec{x} [\exists \vec{y}(\theta_{u}(\vec{y},
    \vec{x})) \wedge \forall\vec{y}(\theta_{u}(\vec{y}, \vec{x})
    \rightarrow \neg \gamma(\vec{y}, \vec{x}))]$.
      
    Rather than working purely syntactically, it is easier to consider
    an arbitrary model $\mathcal{D}$ of the theory $T$.  Since
    $T \vdash \exists\vec{x}_t[\exists \vec{y}_{u}
    \theta_{u}(\vec{y}_{u}, \vec{x}_t) \wedge \forall\vec{y}_{u}(
    \theta_u(\vec{y}_u, \vec{x}_t) \rightarrow \neg \gamma(\vec{y}_u,
    \vec{x}_t))]$,
    there is a $\vec{d}_t \in \mathcal{D}$ of the same length as
    $\vec{x}_t$ and a $\vec{d}'$ of the same length as $\vec{y}_u$ so
    that $\mathcal{D} \models \theta_u(\vec{d}', \vec{d}_t)$ and
    $\mathcal{D} \models \forall\vec{y}_{u}(\theta_u(\vec{y}_u,
    \vec{d}_t) \rightarrow \neg \gamma(\vec{y}_u, \vec{d}_t))$.
    (It is possible that there is repetition of elements within or
    between these two tuples of $\mathcal{D}$; for instance, the
    formula $\theta$ may not say that all of the elements in
    $\vec{x}_{t}$ are unequal.)  Next, simply ``regroup'' the elements
    of $\vec{d}_t$ and $\vec{d}'$ to get new tuples $\vec{b}$ and
    $\vec{b}'$ in $\mathcal{D}$ of the length of $\vec{x}$ and
    $\vec{y}$, respectively, in the algorithm.  (We are not talking
    about any deep re-arrangement here; we're just looking at what
    elements of $\mathcal{D}$ are substituted for what variables in
    $\theta_u$ and $\gamma$.  Again, repetition of elements within
    and/or between the tuples $\vec{b}$ and $\vec{b}'$ may occur.)
    Notice, since $\vec{x}$ is at least as long as $\vec{x}_t$, that
    $\vec{b}$ contains all of $\vec{d}_t$, and possibly more.
    Clearly, since
    $\mathcal{D} \models \theta_u(\vec{d}', \vec{d}_t)$,
    $\mathcal{D} \models \exists \vec{y} \theta_u(\vec{y}, \vec{b})$.
    Now assume that there is $\vec{b}'' \in \mathcal{D}$ of the same
    length as $\vec{y}$ such that
    $\mathcal{D} \models (\theta_u(\vec{b}'', \vec{b}) \wedge
    \gamma(\vec{b}'', \vec{b}))$.
    But since $\vec{b}$ contains all of $\vec{d}_t$, if we simply make
    the ``reverse'' regrouping of $\vec{b}'', \vec{b}$ to get
    $\vec{d}^{*}, \vec{d}_t$, then we'd have
    $\mathcal{D} \models (\theta(\vec{d}^{*}, \vec{d}_{t}) \wedge
    \gamma(\vec{d}^{*}, \vec{d}_t))$,
    which contradicts the fact that
    $\mathcal{D} \models \forall\vec{y}_{u}[\theta_u(\vec{y}_u,
    \vec{d}_t) \rightarrow \neg \gamma(\vec{y}_u, \vec{d}_t)]$.
    Hence, the assumption of the existence of $\vec{b}''$ is false.
    That is,
    $\mathcal{D} \models \exists \vec{y} \theta_u(\vec{y}, \vec{b})$
    and
    $\mathcal{D} \models \forall\vec{y}[\theta_u(\vec{y}, \vec{b})
    \rightarrow \neg \gamma(\vec{y}, \vec{b})]$.
    And so,
    $\mathcal{D} \models \exists \vec{x} [\exists
    \vec{y}(\theta_{u}(\vec{y}, \vec{x})) \wedge
    \forall\vec{y}(\theta_{u}(\vec{y}, \vec{x}) \rightarrow \neg
    \gamma(\vec{y}, \vec{x}))]$.
    Since $\mathcal{D}$ was an arbitrary model of $T$, we can conclude
    that
    $T \vdash \exists \vec{x} [\exists \vec{y}(\theta_{u}(\vec{y},
    \vec{x})) \wedge \forall\vec{y}(\theta_{u}(\vec{y}, \vec{x})
    \rightarrow \neg \gamma(\vec{y}, \vec{x}))]$.
     
    Now, then, we must ask why $\delta_{u+1}$ was defined to be
    $\gamma \wedge (c_{u+1} = c_{u+1})$.  It cannot be that the
    algorithm stopped and exited at Step (3), for then, as noted
    above, the statement we're trying to prove would be true.
    Moreover, by the assumptions about stage $s$, the index of $\Phi$
    is small enough that $\Phi$ will be considered in the
    \textit{first iteration} of the algorithm at Step (6), since
    $u+1 > s$.  Therefore, $\Phi$ would be considered at Step (6) of
    \textit{all iterations} of the algorithm at stage $u+1$ unless the
    algorithm re-defines $i^{*}$ and exits the algorithm in order to
    completely satisfy a higher priority requirement.  But by the
    assumption about stage $s$, all higher priority requirements that
    will ever be completely satisfied already have been completely
    satisfied.  Therefore, no higher priority requirement at stage
    $u+1$ (or any later stage) can be not completely satisfied and
    meet one of the conditions in Step (6).  And again, as noted
    above, it cannot be that the algorithm exits at Step (9) for the
    sake of $\Phi$.  Consequently,
    $\mathcal{A} \models \exists \vec{y}[\theta_u(\vec{y},
    dom(\Phi_{u+1})) \wedge \gamma(\vec{y}, dom(\Phi_{u+1}))]$
    and
    $\mathcal{A} \models \exists \vec{y}[\theta_u(\vec{y},
    dom(\Phi_{u+1})) \wedge \neg\gamma(\vec{y}, dom(\Phi_{u+1}))]$.
    Moreover, in the above paragraph, we saw that
    $T \vdash \exists \vec{x} [\exists \vec{y}(\theta_{u}(\vec{y},
    \vec{x})) \wedge \forall\vec{y}(\theta_{u}(\vec{y}, \vec{x})
    \rightarrow \neg \gamma(\vec{y}, \vec{x}))]$.
    Therefore, at the iteration of the algorithm with this particular
    $\sigma^{*}$, $\Phi$ \textit{does} satisfy condition (b) under the
    second bullet point of Step (6).  And since no higher priority
    requirements become completely satisfied at stage $u+1$, this
    means that $\delta_{u+1}$ should NOT have been defined to be
    $\gamma \wedge (c_{u+1} = c_{u+1})$.  Instead, $\delta_{u+1}$
    should have been defined to be
    $\neg \gamma \wedge \forall \vec{y}(\theta_u(\vec{y},
    ran(\Phi_{u+1})) \rightarrow \neg \gamma(\vec{y},
    ran(\Phi_{u+1}))) \wedge (c_{u+1} = c_{u+1})$.
    But then $R_{\Phi}$ would become completely satisfied at stage
    $u+1$ and hence would never again require attention. This is a
    contradiction.  Therefore, the additional assumption must be
    false.  That is,
    $T \vdash [\phi(\vec{x}_t) \rightarrow
    \exists\vec{y}_{u+1}\theta_{u+1}(\vec{y}_{u+1}, \vec{x}_t)]$.
  \end{proof}

\end{lem}

If we continue all of the notation from the previous lemma, then
almost instantly we obtain the following as a corollary:

\begin{cor} \label{16}  For all $t \geq s$, and for all
  $\rho(\vec{x}_t)$ in the original language,
    
  \begin{enumerate}
    
  \item $T \vdash \phi(\vec{x}_t) \rightarrow \rho(\vec{x}_t)$ if
    $\mathcal{A} \models \rho(\vec{a}_t)$ and
    
  \item $T \vdash \phi(\vec{x}_t) \rightarrow \neg\rho(\vec{x}_t)$ if
    $\mathcal{A} \models \neg \rho(\vec{a}_t)$
    
  \end{enumerate}
    
  Therefore, for each $\vec{a}_t$, $h_t(\vec{a}_t) := \phi(\vec{x}_t)$
  is a complete formula.

  \begin{proof} Fix $t \geq s$ and $\rho(\vec{x}_t)$ in the original
    language.  Note that the $\sigma_e$ enumerate all sentences in the
    expanded language, and for each $e$, there is a $u$ so that
    $\pm\sigma_e$ is one of the conjuncts of $\theta_u$.  Therefore,
    there is some $u \geq t$ such that
    $\exists \vec{y}_{u} \theta_u(\vec{y}_u, \vec{x}_t)$ looks like
    $\exists \vec{y}_{u}[\ldots \wedge \rho(\vec{x}_t) \wedge \ldots]$
    or like
    $\exists \vec{y}_{u}[\ldots \wedge \neg \rho(\vec{x}_t) \wedge
    \ldots]$.  Now apply the conclusion of the previous lemma.
  \end{proof}

\end{cor}

Finally, note that $h$ is not initialized at any stage $t \geq s$,
and, by assumption, $R_{\Phi}$ requires attention infinitely often.
Therefore, by definition of requiring attention, if $|\mathcal{A}|$ is
finite, then $|\mathcal{A}| \subseteq dom(\Phi_t)$ for some
$t \geq s$.  If, instead, $|\mathcal{A}|$ is infinite, then, by
definition, for each stage $t \geq s$ where $R_{\Phi}$ requires
attention, $dom(\Phi_t)$ includes an \textit{initial segment} of the
universe of $\mathcal{A}$ that includes all tuples in the domain of
$h_{t-1}$ and at least one more element.  And by construction, at a
stage $t \geq s$ where $R_{\Phi}$ requires attention, $h_t$ is defined
on $dom(\Phi_t) = \vec{a}_t$ (thought of as a tuple of elements).
Therefore, whether $|\mathcal{A}|$ is finite or infinite, for every
tuple $\vec{a}$ in $\mathcal{A}$, there is a $t \geq s$ so that
$\vec{a} \subseteq \vec{a}_t$.  This fact and the previous corollary
combine to demonstrate that $\mathcal{A}$ is effectively atomic.  This
concludes the proof of Theorem~\ref{Main}.  \hfill $\Box$

\begin{proof}[Proof of Corollary~\ref{three}]
  Note that once we have fixed a requirement $R_\Phi$, a stage $s^*$
  and a stage $s$ as in Remark~\ref{remark}, the above construction
  produces the needed $h$ such that $h(\vec{a})$ is the complete
  formula for $\vec{a}$.  The $h$ is constructed uniformly in our
  model $\mathcal{A}$, a requirement $R_\Phi$, a stage $s^*$ and a
  stage $s$. We can think of latter three items as coded by $e$.
  Hence the construction defines a computable $\Psi$ such that
  $\Psi(\mathcal{A}, e)$ is (the code for) the corresponding $h$.  So
  either there is a requirement $R_\Phi$, a stage $s^*$ and a stage
  $s$ as in Remark~\ref{remark}, which are then coded by $e$, and
  $\Psi(\mathcal{A},e)$ is the computable function witnessing that
  $\mathcal{A}$ is effectively atomic; or there is a decidable
  $\mathcal{M} \models T$, such that there is no computable elementary
  embedding of $\mathcal{A}$ into $\mathcal{M}$.
\end{proof}

\section{Implications in Reverse Mathematics}

The main theorem of this paper, Theorem~\ref{Main}, is that
Effectively Prime $\Rightarrow$ Effectively Atomic.  In the context of
Reverse Mathematics, or, more precisely, in some model of second order
arithmetic, to say a model $\mathcal{A}$ of a theory $T$ is
``effectively prime'' is really just to say that it is prime inside
the model of second order arithmetic; that is, the necessary
embeddings establishing that $\mathcal{A}$ is prime must be among the
functions of the model of second order arithmetic.  

However, as we have stressed above, to say that $\mathcal{A}$ is
effectively atomic is not the same as saying that it is atomic,
because the definition of ``atomic'' does not include the existence of
a single function that ``picks out'' a complete formula for each
tuple.  By ``effectively atomic'' in a model of second order
arithmetic we mean that the function picking out the complete formulas
exists inside this model of second order arithmetic.

The theorem's more technical statement is that, given any
$\emptyset$-decidable $\mathcal{A}$, i.e., a structure whose complete
diagram is computable ($\leq_T \emptyset$), there is a
$\emptyset$-decidable $\mathcal{M}$ such that either, for all
$\Phi' \leq_T \emptyset$, $\Phi'$ does not witness that
$\mathcal{A} \prec \mathcal{M}$, or there is a $h \leq_T \emptyset$
witnessing that $\mathcal{A}$ is effectively atomic.  In the
construction, we used that $\varphi_e$ is a listing, computable (in
$\emptyset$), of all functions that are partial computable (in
$\emptyset$). This basic fact follows immediately from the Enumeration
Theorem. In fact, every possible $\Phi'$ appears as infinitely many
$\varphi_e$, and so, we could use this listing to try to diagonalize
against all $\Phi'$.  If we were able to diagonalize against all
$\Phi'$, then we would have that $\mathcal{A}$ is not effectively
prime.  Otherwise, if we were not, then $\mathcal{A}$ would be
effectively atomic. So, the construction is a ``failed'' priority
argument.

\begin{cor} \label{rca_topped}Effectively Prime $\Rightarrow$
  Effectively Atomic holds in all topped models of RCA$_0$, i.e., all
  models containing a set $X$ in which all other sets are computable.
  Hence, Prime Uniqueness holds in all topped models of
  RCA$_0$.\end{cor}

\begin{proof}

  First we will consider only standard models. Relativizations of the
  first statement in the above paragraph and Enumeration Theorem
  replace the $\emptyset$ with the set $X$ and both relativizations
  remain true.  Therefore, we immediately conclude that Effectively
  Prime $\Rightarrow$ Effectively Atomic holds in all standard, topped
  models of RCA$_0$.  

  Since the relativized Enumeration Theorem holds in RCA$_0$, a
  careful analysis of the proof and its induction arguments is needed
  for nonstandard topped models.  The key is that $\Sigma_1$ bounding
  and bounded $\Sigma_1$ comprehension holds in RCA$_0$.  The fact
  that $\Sigma_1$ bounding holds in $RCA_0$ is well known. Recall that
  bounded $\Sigma_1$ comprehension is for all $\Sigma_1$ formulas,
  $\varphi(x)$, and all $k$, there is a an $Z$ such that $i \in Z $
  iff $i < k $ and $\varphi(i)$.  For details of why bounded
  $\Sigma_1$ comprehension holds in $RCA_0$ see Theorem~II.3.9 of
  \cite{Simpson:98}.  There are a few places where these concepts are
  used.

  The first is to show $\delta_s$ exists and our algorithm at each
  stage terminates.  For $l \leq s$, it is $\Sigma_1$ in RCA$_0$ to
  determine if during stage $s$ there is a substage (a loop though the
  algorithm) where $i^* =l$. This $\Sigma_1$ formula in RCA$_0$ says
  that there is a series of formulas (in our fixed language) and
  substages such that this series witness that $i^* =l$.  This
  $\Sigma_1$ formula needs to be coded carefully using some type of
  course of values recursion.  By bounded $\Sigma_1$ comprehension the
  finite set $X$ of such $l$ exists.  Hence is possible to find the
  least $l$ where $l=i^*$ and therefore $\delta_s$ exists.

  The second place where $\Sigma_1$ bounding and bounded $\Sigma_1$
  comprehension is used is in Remark~\ref{remark} to show a
  requirement $R_\Phi$, a stage $s^*$ and a stage $s$ as in
  Remark~\ref{remark} exist.  Assume that there is a computable
  elementary embedding of $\mathcal{A}$ into $\mathcal{M}$.  Let
  $\Phi$ be any (but not necessarily the least) witness of this
  embedding.  So $R_\Phi$ will require attention infinitely often. A
  requirement being completely satisfied is $\Sigma_1$. By bounded
  $\Sigma_1$ comprehension and $\Sigma_1$ bounding, there is a stage
  $s'$ where every requirement with higher priority than $R_\Phi$
  which is going to be satisfied will be satisfied by stage $s'$. Now
  it is straightforward to find $s^*\geq s'$ and $s$ as in the Remark.

  We also need $\Sigma^0_1$ induction to ensure that if $R_\Phi$
  requires attention infinitely often then $dom(\Phi)$ is
  $|\mathcal{A}|$, see the paragraph after the proof of
  Corollary~\ref{16}. Consider the set of $l$ such that there is stage
  $s$ where the length of the largest initial segment included in
  $dom(\Phi_t)$ is greater than $l$.  This is a $\Sigma^0_1$ definable
  cut and hence $\mathbb{N}$.

  Therefore, Effectively Prime $\Rightarrow$ Effectively Atomic holds
  in all non-standard, topped models of RCA$_0$, as well. 
\end{proof}

  Theorem~\ref{Main} does not necessarily hold in a non-topped model
  of RCA$_0$.  The use of the top $X$ was essential in the above
  proof.  We are grateful to Richard Shore and Leo Harrington for this
  observation and for pointing it out to us.

  In fact, the following example shows that Effectively Prime
  $\Rightarrow$ Effectively Atomic does not always hold.  We thank
  David Belanger for this observation which is connected to his paper
  \cite{MR3248791}.  

  \begin{lem} \label{david}
    Let $\mathcal{S}$ be a Scott Set such that for some
    $X \in \mathcal{S}$, $X' \notin \mathcal{S}$, then ``Effectively
    Prime $\Rightarrow$ Effectively Atomic'' does not hold in
    $\mathcal{S}$ (when $\mathcal{S}$ is viewed as the second order
    part of a standard model of second order arithmetic).
  \end{lem}

  \begin{proof}
    Let $T$ be the theory from Proposition~\ref{theory} relativized to
    the above $X$.  Let $\mathcal{M}$ be a countable model of $T$ in
    $\mathcal{S}$.  $\mathcal{M}$'s isomorphism class is determined by
    the number of elements which realize the non principal type
    $p(x) = \{ \neg R_i(x) | i \in \omega\}$.  The prime model
    $\mathcal{A}$ has no elements realizing this type.
    $\mathcal{A} \in \mathcal{S}$ since $\mathcal{A}$ can be
    computably built from $X$. Computably in $\mathcal{M}$ we can find
    two distinct elements, $x^\mathcal{M}_{i,1}, x^\mathcal{M}_{i,2}$
    realizing $R_i(x)$ in $\mathcal{M}$.  A function computing the
    complete formulas for $x^{\mathcal{A}}_{i,1}$ is not in
    $\mathcal{S}$ since such a function computes $X'$.

    Let $Tr \subseteq 2^{<\omega}$ be the set of $\sigma$ such that
    for all $i, s \leq |\sigma|$,
    $R^\mathcal{A}_{i,s}(x^\mathcal{A}_{i,1 })$ iff
    $R^\mathcal{M}_{i,s}(x^\mathcal{M}_{i,\sigma(i)})$.  $Tr$ is
    computable in $\mathcal{A} \oplus \mathcal{M} \oplus X$ and has at
    least one node at each level $s$. Therefore in $\mathcal{S}$ there
    is an $f \in [Tr]$.  For such an $f$, the types of
    $x^\mathcal{A}_{i,1}$ and $x^\mathcal{M}_{i,f(i)}$ are the same
    for each $i$ in $\mathbb{N}$ and hence the map sending
    $x^\mathcal{A}_{i,1}$ to $x^\mathcal{M}_{i,f(i)}$ can be
    computably (in $f$ and $\mathcal{M}$) extended into an
    embedding. This embedding is also in $\mathcal{S}$.

    So $\mathcal{A}$ is effectively prime in $\mathcal{S}$ but not
    effectively atomic in $\mathcal{S}$.  Note that if
    $\mathcal{B} \in \mathcal{S}$ is also prime then a similar
    argument shows that there is an isomorphism between $\mathcal{A}$
    and $\mathcal{B}$ in $\mathcal{S}$.  So $\mathcal{A}$ is not part
    of a counterexample to effectively
  \end{proof}

  \begin{cor}
    WKL$_0 \wedge \neg$ ACA$_0$ implies the negation of ``Effectively
    Prime $\Rightarrow$ Effectively Atomic''. So ``Effectively Prime
    $\Rightarrow$ Effectively Atomic'' implies
    ACA$_0 \vee \neg $WKL$_0$.
  \end{cor}

  \begin{question}
    What is the reverse mathematics strength of ``Effectively Prime
    $\Rightarrow$ Effectively Atomic''?
  \end{question}

  We know that Prime Uniqueness holds in topped models of RCA$_0$ by
  Corollary~\ref{rca_topped}.  When the $1$-types determine all types,
  the construction from Myhill's Isomorphism Theorem produces an
  isomorphism between $\mathcal{A}$ and $\mathcal{B}$ from the two
  embeddings. However we do not even know whether Prime Uniqueness
  fails in some Scott Set for more complicated theories.

  \begin{question}
    Does Prime Uniqueness hold in RCA$_0$? in WKL$_0$?  What is the
    reverse mathematics strength of Prime Uniqueness?
  \end{question}


\end{document}